\documentclass[11pt,a4paper]{amsart}
\usepackage{amsmath}
\usepackage{amssymb}
\usepackage{amsthm}
\usepackage{amsfonts}
\usepackage{mathrsfs}

\usepackage[left=3cm,top=3cm,right=3cm]{geometry}

\usepackage{color}
\usepackage{enumitem}

\usepackage[colorlinks=true, linkcolor=magenta, citecolor=magenta]{hyperref}

\usepackage[T1]{fontenc}
\usepackage[latin1]{inputenc}
\usepackage[english]{babel}
\usepackage{comment}

\newtheorem{thm}{Theorem}[section]
\newtheorem{lemma}[thm]{Lemma}

\newtheorem{coro}[thm]{Corollary}

\theoremstyle{remark}           %\newtheorem*{remark}{Remark}
\newtheorem{remark}[thm]{Remark}
\newtheorem{example}[thm]{Example}

\newcommand{\diam}{\operatorname{diam}}

\newcommand{\dist}{\operatorname{dist}}

\newcommand{\R}{\mathbb R}
\newcommand{\Om}{\Omega}
\newcommand{\N}{\mathbb N}
\newcommand{\Z}{\mathbb Z}

\newcommand{\W}{\mathcal W}

\newcommand{\M}{{\mathcal M}}

\newcommand{\sub}{\subset}

\newcommand{\Char}[1]{\chi_{\lower 1.5pt\hbox{$\scriptscriptstyle #1$}}}

\newcommand{\ch}[1]{{\mbox{\raise 1pt\hbox{\Large$\chi$}}}_{\lower 1pt\hbox{$\scriptstyle #1$}}}
\newcommand{\capp}[1]{\operatorname{cap}_{#1}}

\def\vint{\mathop{\mathchoice%
          {\setbox0\hbox{$\displaystyle\intop$}\kern 0.22\wd0%
           \vcenter{\hrule width 0.6\wd0}\kern -0.82\wd0}%
          {\setbox0\hbox{$\textstyle\intop$}\kern 0.2\wd0%
           \vcenter{\hrule width 0.6\wd0}\kern -0.8\wd0}%
          {\setbox0\hbox{$\scriptstyle\intop$}\kern 0.2\wd0%
           \vcenter{\hrule width 0.6\wd0}\kern -0.8\wd0}%
          {\setbox0\hbox{$\scriptscriptstyle\intop$}\kern 0.2\wd0%
           \vcenter{\hrule width 0.6\wd0}\kern -0.8\wd0}}%
          \mathopen{}\int}                  

{\catcode`p =12 \catcode`t =12 \gdef\eeaa#1pt{#1}}      % Get slantfactor
\def\accentadjtext#1{\setbox0\hbox{$#1$}\kern   % Convert it with height
                \expandafter\eeaa\the\fontdimen1\textfont1 \ht0 }
\def\accentadjscript#1{\setbox0\hbox{$#1$}\kern % Convert it with height
                \expandafter\eeaa\the\fontdimen1\scriptfont1 \ht0 }
\def\accentadjscriptscript#1{\setbox0\hbox{$#1$}\kern   % Convert it with height
                \expandafter\eeaa\the\fontdimen1\scriptscriptfont1 \ht0 }
\def\accentadjtextback#1{\setbox0\hbox{$#1$}\kern       % Convert it with height
                -\expandafter\eeaa\the\fontdimen1\textfont1 \ht0 }
\def\accentadjscriptback#1{\setbox0\hbox{$#1$}\kern     % Convert it with height
                -\expandafter\eeaa\the\fontdimen1\scriptfont1 \ht0 }
\def\accentadjscriptscriptback#1{\setbox0\hbox{$#1$}\kern % Convert it with height
                -\expandafter\eeaa\the\fontdimen1\scriptscriptfont1 \ht0 }
\def\itoverline#1{{\mathsurround0pt\mathchoice
        {\rlap{$\accentadjtext{\displaystyle #1}
                \accentadjtext{\vrule height1.593pt}
                \overline{\phantom{\displaystyle #1}
                \accentadjtextback{\displaystyle #1}}$}{#1}}
        {\rlap{$\accentadjtext{\textstyle #1}
                \accentadjtext{\vrule height1.593pt}
                \overline{\phantom{\textstyle #1}
                \accentadjtextback{\textstyle #1}}$}{#1}}
        {\rlap{$\accentadjscript{\scriptstyle #1}
                \accentadjscript{\vrule height1.593pt}
                \overline{\phantom{\scriptstyle #1}
                \accentadjscriptback{\scriptstyle #1}}$}{#1}}
        {\rlap{$\accentadjscriptscript{\scriptscriptstyle #1}
                \accentadjscriptscript{\vrule height1.593pt}
                \overline{\phantom{\scriptscriptstyle #1}
                \accentadjscriptscriptback{\scriptscriptstyle #1}}$}{#1}}}}

\newcommand{\iol}{\itoverline}

\begin{document}

\title[Hardy--Sobolev inequalities and capacities]
{Hardy--Sobolev inequalities and weighted capacities \\ in metric spaces}

\author[L. Ihnatsyeva]{Lizaveta Ihnatsyeva}   %  can use \and
\address[L.I.]{Department of Mathematics, Kansas State University, Manhattan, KS 66506, USA}
\email{ihnatsyeva@math.ksu.edu}

\author[J. Lehrb\"ack]{Juha Lehrb\"ack}   %  can use \and
\address[J.L.]{Department of Mathematics and Statistics, P.O. Box 35, FI-40014 University of Jyvaskyla, Finland}
\email{juha.lehrback@jyu.fi}

\author[A. V. V\"ah\"akangas]{Antti V. V\"ah\"akangas}
\address[A.V.V.]{Department of Mathematics and Statistics, P.O. Box 35, FI-40014 University of Jyvaskyla, Finland}
 \email{antti.vahakangas@iki.fi}

\begin{abstract}
Let $\Omega$ be an open set in a metric measure space $X$.
Our main result gives an equivalence between the validity
of a weighted Hardy--Sobolev inequality in $\Omega$ 
and quasiadditivity of a weighted capacity with
respect to Whitney covers of $\Omega$. 
Important ingredients in 
the proof include the use of a discrete
convolution as a capacity test function
and a Maz'ya type characterization of
weighted Hardy--Sobolev inequalities. 
\end{abstract}

\subjclass[2020]{
%	Primary 
%	42B25, %Maximal functions, Littlewood-Paley theory 
%	46E35 ; %Sobolev spaces and other spaces of ''smooth'' functions, embedding theorems, trace theorems
%	Secondary 
	26D10, %Inequalities involving derivatives and differential and integral operators 
%	28A12, %Contents, measures, outer measures, capacities 
%	31B15, %Potentials and capacities, extremal length and related notions in higher dimensions
	31C15, %Potentials and capacities on other spaces
	31E05, %Potential theory on fractals and metric spaces
 	35A23%, %Inequalities applied to PDEs involving derivatives, differential and integral operators, or integrals
%	35J92, %Quasilinear elliptic equations with $p$-Laplacian
%	42B37  %Harmonic analysis and PDEs
}

\date{\today}

\maketitle

%\begin{document}

\section{Introduction}

Let $\Omega\subsetneq \R^n$ be an open set. 
We say that $(q,p,\beta)$-Hardy--Sobolev inequality holds in $\Omega$,
for $1\le p,q<\infty$ and $\beta\in\R$,
if there exists a constant $C>0$ such that the inequality
\begin{equation}\label{eq:(p,q)-weightedhardy_intro}
\biggl(\int_{\Omega} |u(x)|^q \,d(x,\Omega^c)^{\frac{q}{p}(n-p+\beta)-n} \,dx\biggr)^{\frac{1}{q}}
   \leq C\biggl(\int_{\Omega} |\nabla u(x)|^p \,d(x,\Omega^c)^{\beta}\,dx\biggr)^{\frac{1}{p}}
\end{equation}
is valid for all functions $u\in C_0^{\infty}(\Omega)$,
or equivalently, by approximation, for all Sobolev functions
$u\in W_0^{1,p}(\Omega)$ such that the support of $u$ is a compact subset of $\Omega$.
Here $d(x,\Omega^c):=\dist(x,\Omega^c)$
denotes the distance from $x\in \Omega$ 
to the complement $\Omega^c:=\R^n\setminus\Omega$.
When $p=q$, the inequality in~\eqref{eq:(p,q)-weightedhardy_intro} is
called the (weighted) $(p,\beta)$-Hardy inequality, and for $q = \frac{np}{n-p}$
one obtains a weighted Sobolev inequality, which in the unweighted case $\beta=0$
reduces to the usual Sobolev inequality.
Often it is natural (or even necessary) to restrict the parameters in~\eqref{eq:(p,q)-weightedhardy_intro}
to the range $1\le p\le q\le \frac{np}{n-p}$, giving a scale of inequalities interpolating
between the (weighted) Hardy and Sobolev inequalities. 

In this paper, we consider analogous inequalities 
for the so-called Newtonian (Sobolev) functions
in more general metric spaces,
under the standard assumptions that the space $X$ is equipped with a doubling measure and
supports Poincar\'e inequalities. The norm of the gradient on the right-hand 
side of~\eqref{eq:(p,q)-weightedhardy_intro} is then replaced by
a ($p$-weak, $p$-integrable) upper gradient $g_u$ of $u$.
We are mainly interested in the interplay between
Hardy--Sobolev inequalities and conditions given in terms of (weighted) capacities;
see Section~\ref{sect:HScap} for the definitions. We begin by proving a Maz'ya type characterization
for the validity of Hardy--Sobolev inequalities in Theorem~\ref{t.mazya_char_weighted_qp}.
This is a straight-forward generalization of many earlier results,
but in particular it serves to illustrate that the weighted capacity is a natural tool
in questions related to Hardy--Sobolev inequalities.

Our main results in Section~\ref{sect:qadd} relate Hardy--Sobolev inequalities to
quasiadditivity properties of weighted capacities. Recall that
capacities are (typically) subadditive set functions, which, however, very seldom enjoy full additivity.
Quasiadditivity is a weak converse of the subadditivity, involving a multiplicative constant
and applicable only to certain types of sets, often given in terms of Whitney-type covers or
decompositions. The formulations of the results in the general metric setting are slightly complicated
due to the various parameters, see Section~\ref{sect:qadd}, but 
for an open set $\Omega\subsetneq\R^n$ our main result,
Theorem~\ref{t.weightedhardy_c},
can be stated as follows.

\begin{thm}\label{t.qadd_intro}
Let $1<p\le q\le \frac{np}{n-p}$ and $\beta\in\R$,
and let 
$\mathcal{W}_c(\Omega)=\{B_i:i\in\N\}$
be a cover of an open set $\Omega\subsetneq \R^n$ 
by Whitney balls $B_i=B(x_i,cd(x_i,\Omega^c))$,
with $0<c<1/53$.
Then the following conditions are equivalent:
\begin{enumerate}[label=\textup{(\roman*)}] %,leftmargin=18pt

\item\label{it.weightedhardy_c_HS_intro} 
Hardy--Sobolev inequality \eqref{eq:(p,q)-weightedhardy_intro} holds in $\Omega$.

\item\label{it.weightedhardy_c_compact_QR_intro} There exist constants $C_1$ and $C_2$ such that
the weighted relative capacity %$\capp{p,\beta}(\,\cdot\,,\Omega)$
satisfies the quasiadditivity property
\begin{equation*}
\sum_{i=1}^\infty \capp{p,\beta}(E\cap  B_i,\Omega)^{\frac q p}\le C_1 \capp{p,\beta}(E,\Omega)^{\frac q p},
\end{equation*}
for all sets $E\Subset\Omega$, and 
the capacity lower bound 
\[
\capp{p,\beta}(B_i,\Omega) \ge C_2\, d(x_i,\Omega^c)^{n+\beta-p}
\]
holds for every $i\in \N$.
\end{enumerate}
\end{thm}

Above, we write $E\Subset\Omega$
if the closure $\iol E$ is a compact subset of $\Omega$.
For such sets, 
the weighted relative capacity
is defined by setting
\[
\capp{p,\beta}(E,\Omega):= \inf_u \int_{\Om} \lvert \nabla u(x)\rvert^p\, d(x,\Omega^c)^{\beta}\, dx,
\]
where the infimum is taken over all
quasicontinuous representatives of $u\in W_0^{1,p}(\Omega)$
such that $u(x)\ge 1$ for all $x\in E$ and $u$ has a compact support in $\Omega$.

The case $\beta=0$ of Theorem~\ref{t.qadd_intro},
but with respect to Whitney cube decompositions,
was established in~\cite[Theorem~10.52]{KLVbook}.
The constant $1/53$ in the statement of Theorem~\ref{t.qadd_intro}
is not that essential and it has not been optimized. In fact,
it could be replaced by a larger number (up to $1/3$) by using
a further covering argument.

In general, both conditions in part~\ref{it.weightedhardy_c_compact_QR_intro} 
of Theorem~\ref{t.qadd_intro} are
needed; this is illustrated by examples in Section~\ref{sect:qreg}. However, there are
cases in which the capacity lower bound automatically holds, and therefore only the 
quasiadditivity property needs to be assumed in part~\ref{it.weightedhardy_c_compact_QR_intro}.
This happens, for instance, in the unweighted ($\beta=0$) case in $\R^n$ when $1<p<n$;
see Section~\ref{sect:qreg} for further discussion. 

In the metric space setting,
the case $q=p$, $\beta=0$ of our main result (Theorem~\ref{t.weightedhardy_c}) was
established in~\cite{LehrbackShanmugalingam2014}. The proofs in~\cite{LehrbackShanmugalingam2014} were based on potential theoretic tools
such as Harnack inequalities and the existence of capacitary potentials.
In this paper, we use a different approach which 
combines and develops ideas from~\cite{DydaVahakangas2015,HurriSyrjanenVahakangas2015,KLVbook,LuiroVahakangas2016} and
is better suited to the 
weighted ($\beta\neq 0$) case and to exponents $q\ge p$.
An important feature in
the proof of Theorem~\ref{t.weightedhardy_c} is the use of the so-called discrete convolution
as a capacity test function. 
The definition of the discrete convolution is recalled in
Section~\ref{sect:max}, where we also show that 
if $u$ belongs to the Newtonian space $N_0^{1,p}(\Omega)$,
then a local maximal function of $g_u$ 
is a $p$-weak upper gradient of the discrete convolution $u_t$.
This fact will be utilized in the proof of Theorem~\ref{t.weightedhardy_c}.
Theorem~\ref{t.weightedhardy_c} also contains a weaker variant of
the quasiadditivity property, where instead of arbitrary sets $E\subset\Omega$
the quasiadditivity is tested using only unions of Whitney balls.

If the weight $w(x)=d(x,\Omega^c)^\beta$ is $p$-admissible in the metric space $X=(X,d,\mu)$, that is,
the weighted measure $w\,d\mu$ is doubling and supports a
$(1,p)$-Poincar\'e inequality, then 
the weighted cases of our main results could be obtained
from the corresponding unweighted results applied 
in the weighted metric space $(X,d,w\,d\mu)$.
In particular, then for $q=p$ the weighted case can be directly obtained from the results in~\cite{LehrbackShanmugalingam2014}. 
Many examples of admissible weights are provided by weights from the Muckenhoupt classes, 
see~\cite[Theorem~4]{Bjorn2001} and the remark after~\cite[Theorem~7]{Bjorn2001}.
Sufficient and necessary conditions for the weight $w(x)=d(x,\Omega^c)^\beta$ to be a Muckenhoupt weight
have been given in~\cite{DydaEtAl2019} in terms of the (co)dimension of $\Omega^c$. 
However, we emphasize that our more general approach applies also
in the cases where $w$ is not admissible, and thus we obtain a unified theory for all 
weighted Hardy--Sobolev inequalities. 

In addition to the results in~\cite[Theorem~1]{LehrbackShanmugalingam2014} and~\cite[Theorem~10.52]{KLVbook},  
the quasiadditivity property has been considered 
in the Euclidean space $\R^n$ for instance in~\cite{Aikawa1991,AikawaEssen1996}
for Riesz capacities, and in~\cite{DydaVahakangas2015,LuiroVahakangas2016} ($q=p$) 
and~\cite{HurriSyrjanenVahakangas2015} ($q\ge p$)
for fractional capacities with the help of fractional Hardy(--Sobolev) inequalities.
Sufficient conditions for different versions of (weighted) 
Hardy and Hardy--Sobolev inequalities,
and hence also for the corresponding quasiadditivity, 
have been given for example in
\cite{DydaEtAl2019,KLVbook,Lehrback2017jam,LehrbackVahakangas2016,Mazya1985};
see also the references therein.

The outline for the rest of the paper is as follows.
In Section~\ref{sect:preli} we recall the necessary background on
analysis on metric spaces:
doubling measures, $p$-weak upper gradients, Poincar\'e inequalities, and Whitney covers.
Metric space variants of the
Hardy--Sobolev inequality~\eqref{eq:(p,q)-weightedhardy_intro} and the
weighted relative capacity are introduced in Section~\ref{sect:HScap}, where 
we also prove the Maz'ya type characterization
of weighted Hardy--Sobolev inequalities and show how the capacity lower bound
follows from the Hardy--Sobolev inequality. In Section~\ref{sect:max} we prepare for
our main results by proving the boundedness of a local maximal operator
and showing how this maximal operator can be used to give an upper gradient for the 
discrete convolution.
Section~\ref{sect:qadd} contains our main results on quasiadditivity, and
finally in Section~\ref{sect:qreg} we consider the special case of $Q$-regular
spaces and give examples illustrating the necessity of the conditions
in our characterizations.

As usual, we let $C$ denote positive constants whose exact value may change at each occurrence. 
By $\ch{E}$ we denote the characteristic function of a set $E\subset X$;
that is, $\ch{E}(x)=1$ if $x\in E$ and $\ch{E}(x)=0$ if $x\in X\setminus E$.

\section{Preliminaries}\label{sect:preli}

\subsection{Metric spaces with a doubling measure}
Let $X=(X,d,\mu)$ be a metric measure space. We assume for the rest of this paper that
$\mu$ is a Borel measure on $X$, with $0<\mu(B)<\infty$ whenever $B=B(x,r)$ is an open ball in $X$, and
that $\mu$ is \emph{doubling}, that is, there is a constant $C_\mu$ such that such that
\begin{equation}\label{e.doubling}
\mu(2B) \le C_\mu\, \mu(B)
\end{equation}
for all balls $B$ in $X$. 
Here we use the notation $tB=B(x,tr)$,
when  $0<t<\infty$ and $B=B(x,r)$.

By iterating the doubling condition \eqref{e.doubling}, we can find constants $Q>0$ and $C>0$ such that
\begin{equation}\label{eq:QDoublingDim}
\frac{\mu(B(y,r))}{\mu(B(x,R))}\ge C\Bigl(\frac rR\Bigr)^Q
\end{equation}
whenever $0<r\le R<\diam X$ and $y\in B(x,R)$; see \cite[Lemma~3.3]{BjornBjorn2011}.

\subsection{Upper gradients and Newtonian functions}
Let $1\le p<\infty$.
We say that a $\mu$-measurable function $g\colon X\to [0,\infty]$ is a {\em $p$-weak upper gradient}
of  $u\colon X\to [-\infty,\infty]$, 
if 
\begin{equation}\label{e.modulus}
\lvert u(\gamma(0))-u(\gamma(\ell_\gamma))\rvert \le \int_\gamma g\,ds
\end{equation}
for $p$-almost every curve $\gamma\colon [0,\ell_\gamma]\to X$; that is, there exists a non-negative Borel function 
$\rho\in L^p_{\mathrm{loc}}(X;d\mu)$ such that 
$\int_\gamma \rho\,ds=\infty$ whenever \eqref{e.modulus} does not hold or is not defined.
For a function $u\in L^p(X;d\mu)$, we denote by $\mathcal{D}^{p}(u)$ 
the set of all $p$-weak upper gradients $g\in L^p(X;d\mu)$ of $u$.
We remark that $\mathcal{D}^p(u)$ can be empty.
See~\cite{BjornBjorn2011} and~\cite{HeinonenKoskelaShanmugalingamTyson2015} 
for detailed treatments of $p$-weak upper gradients.

Using $p$-weak upper gradients as a substitute
for modulus of the weak derivative, one defines the norm
\[
  \Vert u\Vert_{N^{1,p}(X)}:= \left(\int_X|u|^p\, d\mu\, +\, \inf_g\int_X g^p\, d\mu\right)^{1/p},
\]
where $1\le p<\infty$ and the infimum is taken over all $g\in\mathcal{D}^p(u)$.  
The \emph{Newtonian space} $N^{1,p}(X)$ is
the set 
\[
  \bigl\{u \colon X \to[-\infty,\infty]\, :\, \Vert u\Vert_{N^{1,p}(X)}<\infty\bigr\}
\] 
equipped with the norm $\Vert \cdot \Vert_{N^{1,p}(X)}$. 
We assume that functions in $N^{1,p}(X)$ are defined everywhere, and
not just up to an equivalence class.

When $\Omega \subset X$ is an open set, we denote by 
$N^{1,p}_{0}(\Omega)$
the space of all Newtonian
functions on $X$ that vanish in 
the complement $\Omega^c= X \setminus \Omega$. 
Moreover, we let $N^{1,p}_{c}(\Omega)$ denote the space of all Newtonian
functions on $X$ whose support, i.e.\ the closure of the set where $u\neq 0$,
is a bounded set having a strictly positive distance to the complement $\Omega^c$.
Then $N^{1,p}_{c}(\Omega)\subset N^{1,p}_{0}(\Omega)$ and $N^{1,p}_{c}(\Omega)$ is a natural class of
test functions in weighted Hardy--Sobolev inequalities. If $X$ is complete,
it is equivalent to require that functions in $N^{1,p}_{c}(\Omega)$
have a compact support in $\Omega$, but in non-complete spaces the above definition
gives more flexibility. 

We refer to
\cite{BjornBjorn2011,HeinonenKoskelaShanmugalingamTyson2015,Shanmugalingam2000} for 
more information on Newtonian spaces. 

\subsection{Poincar\'e inequalities}
We say that the space $X=(X,d,\mu)$ supports a \emph{$(q,p)$-Poincar\'e inequality} for $1\le q,p<\infty$ if there exist constants
$C>0$ and $\lambda\ge 1$ such that
\begin{equation}\label{PoincareInequality}
    \left(\vint_{B} |u-u_B|^q\, d\mu\right)^{\frac{1}{q}} \le C\, r \left(\vint_{\lambda B}g^p\, d\mu\right)^{\frac{1}{p}}
\end{equation}
for all measurable functions $u$ in $X$ and all $p$-weak upper gradients $g$ of $u$,
where the left-hand side of \eqref{PoincareInequality}
is defined to be $\infty$, if the mean value integral
\[
   u_B:=\frac{1}{\mu(B)}\, \int_B u\, d\mu =:\vint_{B}u\, d\mu
\]
is not defined.

The constant $\lambda$ above is called the dilatation constant for the $(q,p)$-Poincar\'e inequality.
The $(1,p)$-Poincar\'e inequality 
for $1\le p<Q$, where $Q$ is as in \eqref{eq:QDoublingDim},  
together with the doubling condition \eqref{e.doubling} implies a $(q,p)$-Poincar\'e  inequality
with $q=Qp/(Q-p)$; see~\cite{BjornBjorn2011,HajlaszKoskela2000}.

If $X$ is complete and supports a $(1,p)$-Poincar\'e inequality for some $1<p<\infty$, then there exists $1<s<p$ such that 
$X$ supports a $(1,s)$-Poincar\'e inequality; see \cite{KeithZhong2008}. Since we do not require the completeness from $X$ in 
this paper, we can not refer to this self-improvement result and 
hence in some of our results we directly assume that $X$ supports a $(1,s)$-Poincar\'e inequality for some 
$1<s<p$.

\subsection{Whitney covers} Let $\Omega\subsetneq X$ be an open set.
When $x\in\Omega$, we denote by 
$d(x,\Omega^c)=\dist(x,\Omega^c)$
the distance from $x\in \Omega$ 
to the complement $\Omega^c:=X\setminus\Omega$.

Let $0<c< 1/3$. 
There exists a countable family
${\mathcal W}_c(\Omega)=\{ B_i : i\in \N \}$ of balls $B_i=B(x_i,r_i)$, 
$r_i=c\,d(x_i,\Omega^c)$, such that
${\mathcal W}_c(\Omega)$ is a cover of $\Omega$
and the balls $B_i$ have a uniformly bounded
overlap, that is, there
exists $1\le C<\infty$ such that
\[
1\le\sum_{i=1}^\infty \ch{B_i}(x)\le C,
\]
for every $x\in\Omega$.
The collection ${\mathcal W}_c(\Omega)$ of
\emph{Whitney balls $B_i$} is called a 
\emph{Whitney cover} of $\Omega$.

The following lemma collects useful properties of Whitney covers. 
Proofs of properties \ref{whitney cover}, \ref{whitney inside U} and \ref{whitney distance to U} 
easily follow from the definition of balls $B_i$ and $B_i^*$, 
and property \ref{whitney overlap} is proved, for instance, in~\cite{BjornBjornShanmugalingam2007}.

\begin{lemma}\label{whitney}
Let $\Omega\subsetneq X$ be an open set and let $L> 1$. 
Suppose ${\mathcal W}_c(\Omega)=\{ B_i : i\in \N\}$ is a family of Whitney balls as above, with $c\le (3L)^{-1}$, and define $B_i^*:=LB_i$. Then the following assertions hold:
%There is $M\in\N$ such that for all $i\in\N$,
\begin{enumerate}[label=\textup{(\roman*)}]
%\item
%\label{whitney disjoint}
%the balls $\frac15B_i$ are disjoint,
\item
\label{whitney cover}
$\Omega=\bigcup_{i\in \N} B_i$,
\item
\label{whitney inside U}
$B_i^*\subset \Omega$, for every $i\in\N$,
\item
\label{whitney distance to U}
if $x\in B_i^*$, then $(1/c-L)r_i\le d(x,\Omega^c) \le (1/c+L)r_i$,
%\item
%\label{whitney x*i}
%there is $x^*_i\in X\setminus \Omega$ such that $d(x_i,x^*_i)<15r_i$,
\item
\label{whitney overlap}
there is $M\in\N$ such that $\sum_{i\in \N}\ch{B_i^*}(x)\le M$ for all $x\in \Omega$.
\end{enumerate}
\end{lemma}

\section{Hardy--Sobolev inequality and relative capacity}\label{sect:HScap}

In $\R^n$, the Hardy--Sobolev inequality~\eqref{eq:(p,q)-weightedhardy_intro}
is written in terms of powers of the distance $d(x,\Omega^c)$. In the metric case, it turns out to be
natural to interpret the part $d(x,\Omega^c)^{n(q-p)/p}$ as the measure
of the ball $B(x,d(x,\Omega^c))$ to the power $(q-p)/p$. This leads to the following
formulation, which in the case $X=\R^n$ is equivalent to~\eqref{eq:(p,q)-weightedhardy_intro}.
See also Section~\ref{sect:qreg} for the case of $Q$-regular metric spaces.

Let $\Omega\subsetneq X$ be an open set.
We say that a $(q,p,\beta)$-Hardy--Sobolev inequality holds in $\Omega$,
for $1\le p,q<\infty$ and $\beta\in\R$,
if there exists a constant $C>0$ such that the inequality
\begin{equation}\label{eq:(p,q)-weightedhardy}
\biggl(\int_{\Omega} \frac{|u(x)|^q}{d(x,\Omega^c)^{\frac{q}{p}(p-\beta)}}\mu(B(x,d(x,\Omega^c)))^\frac{q-p}{p}\,d\mu(x)\biggr)^{\frac{1}{q}}
   \leq C\biggl(\int_{\Omega} g(x)^p \,d(x,\Omega^c)^{\beta}\,d\mu(x)\biggr)^{\frac{1}{p}}
\end{equation}
is valid for all $u\in N_c^{1,p}(\Omega)$
and for all $g\in\mathcal{D}^p(u)$.

We write $E\Subset\Omega$ if the set $E$ is bounded
and $d(E,\Omega^c)>0$, that is, $E$ has a strictly positive distance to the complement $\Omega^c$.
The \emph{relative $(p,\beta)$-capacity} of a set $E\Subset\Omega$,
for $1 \le p<\infty$  and $\beta\in\R$, is defined to be the number
\[
\capp{p,\beta}(E,\Omega):= \inf_u\inf_{g} \int_{\Om} g(x)^p\, d(x,\Omega^c)^{\beta}\, d\mu(x),
\]
where the infimum is taken over all $u\in N_c^{1,p}(\Omega)$ such that $u(x)\ge 1$ for all $x\in E$, and over all $g\in\mathcal{D}^p(u)$. A function $u$ satisfying the above conditions is called a \emph{capacity test function} for $E$. 
Observe that if $E\Subset\Omega$, then there exist capacity test functions for $E$ and 
thus $\capp{p,\beta}(E,\Omega)<\infty$. 
For instance, we can test the capacity with the Lipschitz function 
\[
\varphi(x)=\max\left\{0,1-\frac{2d(x,E)}{d(E,\Omega^c)}\right\}, \qquad x\in X.
\]

The following Maz'ya-type characterization 
is one manifestation of the close connection
between Hardy--Sobolev inequalities
and relative capacities; see \cite{Mazya1985} for the origins of this kind of characterizations.
 
\begin{thm}\label{t.mazya_char_weighted_qp}
Let $1\le p\le q<\infty$, $\beta\in\mathbb{R}$ and let $\Omega\subsetneq X$ be an open set. The $(q,p,\beta)$-Hardy--Sobolev inequality holds in $\Omega$
if and only if there is a constant $C_1$ such that
\begin{equation}\label{eq:mazya and weighted hardy}
 \int_E  \frac{\mu(B(x,d(x,\Omega^c)))^\frac{q-p}{p}}{d(x,\Omega^c)^{\frac{q}{p}(p-\beta)}}\,d\mu(x) \leq C_1 \capp{p,\beta}(E,\Omega)^{\frac{q}{p}},
\end{equation}
for all $E\Subset\Omega$.
\end{thm}

\begin{proof} The main lines of the proof follow the proof of \cite[Theorem~4.1]{KorteShanmugalingam2010} where a similar characterization is obtained for the validity of the $p$-Hardy inequality,  that is the case $p=q$ and $\beta=0$.

First assume that $(q,p,\beta)$-Hardy--Sobolev inequality \eqref{eq:(p,q)-weightedhardy} holds in $\Omega$, with a constant $C>0$.
Let $E\Subset\Omega$,  
and let $u\in N_c^{1,p}(\Omega)$ be such that $u(x)\ge 1$, for every $x\in E$.
Then
\begin{equation}\label{e.applies}
\begin{split}
\int_E\frac{\mu(B(x,d(x,\Omega^c)))^\frac{q-p}{p}}{d(x,\Omega^c)^{\frac{q}{p}(p-\beta)}}\,d\mu(x)
&\leq \int_\Omega
 \frac{|u(x)|^q}{d(x,\Omega^c)^{\frac{q}{p}(p-\beta)}}\mu(B(x,d(x,\Omega^c)))^\frac{q-p}{p}\,d\mu(x)\\
&\leq
  C^q\biggl(\int_\Omega g(x)^p \,d(x,\Omega^c)^{\beta}\,d\mu(x)\biggr)^\frac{q}{p},
\end{split}
\end{equation}
where $g\in\mathcal{D}^p(u)$.
By taking infimum over all $g\in\mathcal{D}^p(u)$, and then over all functions $u$ as above,
we obtain~\eqref{eq:mazya and weighted hardy} with $C_1=C^q$.

Then assume that inequality~\eqref{eq:mazya and weighted hardy} holds with a constant $C_1$ for all sets
$E\Subset\Omega$. Let $u\in N_c^{1,p}(\Omega)$ and 
$g\in\mathcal{D}^p(u)$. 
For $j\in\Z$, define
\[
E_j=\{x\in \Omega : |u(x)|>2^j\}.
\]
Since $E_j\Subset\Omega$ for every $j\in\Z$,
by~\eqref{eq:mazya and weighted hardy} we have
\begin{align*}
&\int_\Omega\frac{|u(x)|^q}{d(x,\Omega^c)^{\frac{q}{p}(p-\beta)}}\mu(B(x,d(x,\Omega^c)))^\frac{q-p}{p}\,d\mu(x)\\
&\qquad\leq \sum_{j=-\infty}^\infty 2^{(j+2)q}\int_{E_{j+1}\setminus E_{j+2}}\frac{\mu(B(x,d(x,\Omega^c)))^\frac{q-p}{p}}{d(x,\Omega^c)^{\frac{q}{p}(p-\beta)}}\,d\mu(x)\\
&\qquad\leq \sum_{j=-\infty}^\infty 2^{(j+2)q}\int_{E_{j+1}}\frac{\mu(B(x,d(x,\Omega^c)))^\frac{q-p}{p}}{d(x,\Omega^c)^{\frac{q}{p}(p-\beta)}}\,d\mu(x)\\
&\qquad\leq 4^q C_1\sum_{j=-\infty}^\infty 2^{jq}\capp{p,\beta}(E_{j+1},\Omega)^\frac{q}{p}.
\end{align*}
For every $j\in\Z$, define $u_j\colon X\to[0,1]$ by
\begin{equation*}
u_j(x)=\begin{cases}
1, & \textrm{if }|u(x)|\geq 2^{j+1},\\
2^{-j}|u(x)|-1, & \textrm{if } 2^{j}<|u(x)|<2^{j+1},\\
0,&\textrm{if }|u(x)|\leq 2^{j}.
\end{cases}
\end{equation*}
Then $u_j\in N_c^{1,p}(\Omega)$ and $u_j=1$ in $E_{j+1}$. From the gluing lemma \cite[Lemma~2.19]{BjornBjorn2011} we obtain that
$g_{j}=2^{-j}g\ch{E_j\setminus E_{j+1}}\in \mathcal{D}^p(u_j)$.
Using $u_j$ as a test function for the weighted capacity $\capp{p,\beta}(E_{j+1},\Omega)$, we obtain
\begin{align*}
\capp{p,\beta}(E_{j+1},\Omega)&\le \int_{\Omega} g_{j}(x)^p\,d(x,\Omega^c)^{\beta}\,d\mu(x)\\
&= \int_{E_{j}\setminus E_{j+1}} g_{j}(x)^p\,d(x,\Omega^c)^{\beta}\,d\mu(x)\le 2^{-jp}
\int_{E_{j}\setminus E_{j+1}} g(x)^p\,d(x,\Omega^c)^{\beta}\,d\mu(x).
\end{align*}
Since $q/p\geq 1$, it follows that
\begin{align*}
\sum_{j=-\infty}^\infty 2^{jq}\capp{p,\beta}(E_{j+1},\Omega)^\frac{q}{p} 
&\le \sum_{j=-\infty}^\infty\biggl(\int_{E_{j}\setminus E_{j+1}} g(x)^p\,d(x,\Omega^c)^{\beta}\,d\mu(x)\biggr)^\frac{q}{p}\\
&\leq \Bigg(\sum_{j=-\infty}^\infty\int_{E_{j}\setminus E_{j+1}}  g(x)^p\,d(x,\Omega^c)^{\beta}\,d\mu(x)\Biggr)^\frac{q}{p}\\
&=\biggl(\int_\Omega g(x)^p\,d(x,\Omega^c)^{\beta}\,d\mu(x)\biggr)^\frac{q}{p}.
\end{align*}
This shows that the $(q,p,\beta)$-Hardy--Sobolev inequality holds with the constant $C=4C_1^{1/q}$ for every function 
$u\in N_c^{1,p}(\Omega)$ and for every $g\in\mathcal{D}^p(u)$.
\end{proof}

Let $B=B(x,r)\in \mathcal{W}_c(\Omega)$, $0<c<1/3$, be a Whitney ball defined as above and set $B^*=LB$ with $L=1/(3c)$. 
In this case $B^*\Subset\Omega$ and we can take the Lipschitz function 
\[
\varphi(x)=\max\left\{0,1-\frac{d(x,B)}{d(B,X\setminus B^*)}\right\}, \qquad x\in X,
\]
as a test function for the capacity $\capp{p,\beta}(B,\Omega)$. Then, by the properties of Whitney balls, 
see Lemma~\ref{whitney}, and by the doubling condition for measure $\mu$, it follows that
there exists a constant $C$, depending on the parameter $c$ 
and the doubling constant $C_\mu$, such that
\[
\capp{p,\beta}(B,\Omega)\le C\mu(B)r^{\beta-p},
\]
for every ball $B=B(x,r)\in \mathcal{W}_c(\Omega)$. 

By the following lemma, if a $(q,p,\beta)$-Hardy--Sobolev inequality holds in $\Omega$, then the corresponding lower bound for the capacity is valid as well, and, therefore, the value of $\capp{p,\beta}(B,\Omega)$ is comparable to $\mu(B)r^{\beta-p}$
for every ball $B=B(x,r)\in \mathcal{W}_c(\Omega)$. 

\begin{lemma}\label{eq.LowerBallweightedCapacity}
Let %$1<p<\infty$, 
$1\le p, q<\infty$ and $\beta\in\R$, and assume that a $(q,p,\beta)$-Hardy--Sobolev inequality~\eqref{eq:(p,q)-weightedhardy} holds in an open set $\Omega\subsetneq X$. Then
there exists a constant $C>0$ such that
\begin{equation}\label{eq.weightedcap_lower_w}
\capp{p,\beta}(B,\Omega) \ge C \mu(B)r^{\beta-p}
\end{equation}
for every ball $B=B(x,r) \in\mathcal{W}_c(\Omega)$.
\end{lemma}

\begin{proof}
By applying the estimate in~\eqref{e.applies} (which holds also if $q<p$)
and the property \ref{whitney distance to U} of Whitney balls, we have
\begin{align*}
\capp{p,\beta}(B,\Omega)^{\frac{q}{p}}&\ge C\int_B  \frac{\mu(B(y,d(y,\Omega^c)))^\frac{q-p}{p}}{d(y,\Omega^c)^{\frac{q}{p}(p-\beta)}}\,d\mu(y)\\
&\ge C \int_B  \mu(B)^\frac{q-p}{p}r^{\frac{q}{p}(\beta-p)}\,d\mu(y)=C\big(\mu(B)r^{\beta-p}\big)^\frac{q}{p}.
\end{align*}
The claim follows by raising this to power $p/q>0$.
\end{proof}

\section{Local maximal functions and discrete convolutions}\label{sect:max}

Let $\Omega\subsetneq X$ be an open set and let $0<\kappa<1$.
For a measurable function $f$ on $\Omega$, define the \emph{local maximal function}
\[
 {\M}_{\Omega,\kappa}f(x) =
\sup_{0<r<\kappa d(x,\Omega^c)} \vint_{B(x,r)} \lvert f\rvert\,d\mu,\qquad x\in \Omega.
\]
For convenience, we set ${\M}_{\Omega,\kappa}f(x)=0$ if $x\in X\setminus\Omega$.
Recall that the usual (centered)
Hardy--Littlewood maximal operator $\M$ is defined for all $x\in X$ 
by using the integral averages as in the definition of
${\M}_{\Omega,\kappa}$ but omitting the upper bound $r<\kappa d(x,\Omega^c)$
for the radii. Hence
${\M}_{\Omega,\kappa}f(x)\le \M f(x)$
whenever $f$ is measurable on $X$ and $x\in X$.

It is well known that $\M$ is bounded on $L^s(X,d\mu)$ for all $1<s<\infty$; 
see, for instance, \cite[Section~3.2]{BjornBjorn2011}.
The next lemma gives a similar boundedness
result for the local maximal operator ${\M}_{\Omega,\kappa}$
on the weighted space $L^s(\Omega;w_\beta\,d\mu)$,
when $1<s<\infty$, $0<\kappa<1/5$
and $w_\beta(x)=d(x,\Omega^c)^\beta$
is a distance weight, for $\beta\in\R$ and $x\in \Omega$.

\begin{lemma}\label{l.maximal_bounded}
Let $0<\kappa<1/5$ and $1<s<\infty$, and let $\Omega\subsetneq X$ be an open set. 
Let $\beta\in\R$ and define $w_\beta(x)=d(x,\Omega^c)^\beta$, for $x\in\Omega$.
Then ${\M}_{\Omega,\kappa}$ is
bounded on $L^s(\Omega;w_\beta\,d\mu)$, that is, there is a constant 
$C$ such that
\[
\int_{\Omega} \bigl({\M}_{\Omega,\kappa} f\bigr)^{s} w_\beta\,d\mu
\le C\int_{\Omega} \lvert f\rvert^s w_\beta\,d\mu,
\]
for every $f\in L^s(\Omega;w_\beta\,d\mu)$. 
\end{lemma}

\begin{proof} 
Let ${\mathcal W}_c(\Omega)=\{ B_i:i\in \N\}$ 
be a Whitney cover of $\Omega$, with $c=1/9$,
and take $L=3$.
Suppose $x\in B_i=B(x_i,r_i)$. If $y\in B(x,\kappa d(x,\Omega^c))$, then
\[
d(y,x_i)\le r_i+\kappa d(x,\Omega^c) < r_i+\tfrac{1}{5}(r_i+9r_i)=3r_i,
\]
which implies that
$
B(x,\kappa d(x,\Omega^c))\subset B(x_i,3r_i)= B_i^*.
$
Hence, by the definition of the local maximal function, for every $x\in B_i$, 
\[
{\M}_{\Omega,\kappa}(f)(x)={\M}_{\Omega,\kappa}(f\ch{B_i^*})(x)
\le {\M}(f\ch{B_i^*})(x).
\]
This observation, together with the properties of Whitney balls
and the boundedness of $\M$ on $L^s(X,d\mu)$, gives
\begin{align*}
&\int_\Omega ({\M}_{\Omega,\kappa} f(x))^s w_\beta(x) d\mu(x)
\le \sum_i \int_{B_i} ({\M}_{\Omega,\kappa} f(x))^s d(x,\Omega^c)^\beta d\mu(x)\\
&\qquad\le C\sum_i r_i^\beta \int_{B_i} ({\M}(f\ch{B_i^*})(x))^s d\mu(x)
\le C\sum_i r_i^\beta \int_X |f(x)|^s\ch{B_i^*}(x) d\mu(x)\\
&\qquad\le C\sum_i \int_{B_i^*} |f(x)|^s d(x,\Omega^c)^\beta d\mu(x) 
\le C\int_\Omega |f(x)|^s w_\beta(x) d\mu(x).
\end{align*}
This proves the claim. 
\end{proof}

Let $0<t<1$ be a scaling parameter and let ${\mathcal W}_c(\Omega)=\{ B_i:i\in \N\}$ be a Whitney cover 
of an open set $\Omega\subsetneq X$ with $c=t/18$. There is a sequence of Lipschitz functions $\{\varphi_i\}_{i\in \N}$ 
corresponding to the cover ${\mathcal W}_c(\Omega)$ such that $\varphi_i$ satisfies the following properties
for every $i\in\N$: $0\le\varphi_i\le 1$, $\varphi_i=0$ outside of the balls $6B_i$, $\varphi_i\geq \nu>0$ on $3B_i$, 
and $\varphi_i$ is Lipschitz with constant $K/r_i$, where $\nu$ and $K$ 
only depend on the doubling constant of the measure $\mu$. Moreover,
\[
\sum_{i=1}^{\infty}\varphi_i(x)=1
\]
for every $x\in \Omega$.
Then the \emph{discrete convolution} of a locally integrable function $u$ at the scale $t$ is
\begin{equation}\label{DiscreteConvolution}
u_t(x)=\sum_{i=1}^{\infty}\varphi_i(x)u_{3B_i}, \qquad x\in X.
\end{equation}
Notice that $u_t(x)\le C \M u(x)$, for every $x\in X$, 
where $\M$ is the standard Hardy--Littlewood maximal operator on $X$. 
Hence, the boundedness of $\M$ 
implies that if $1<s<\infty$ and $u\in L^s(X,d\mu)$, then $u_t\in L^s(X,d\mu)$ as well.  

The proof of the following lemma is strongly inspired by the proof of~\cite[Lemma~5.1]{AaltoKinnunen2010}.
However, since Lemma~\ref{MaxFunctIsUpperGradient} contains an additional parameter $\kappa$,
we present the details for the convenience of the reader.

\begin{lemma}\label{MaxFunctIsUpperGradient}
Let $1<p<\infty$ and $0<\kappa\le 1$, and let
$\Omega\subsetneq X$ be an open set.
Assume that $X$ supports a $(1,s)$-Poincar\'{e} inequality for some $1<s<p$ with a dilation constant $\lambda$, and let $0<t<\min\{1,6\kappa/(3\lambda+2\kappa)\}$. Then 
there is a constant $C$ such that if $u\in  N^{1,p}_0(\Omega)$ and $g\in\mathcal{D}^p(u)$,
then $C({\M}_{\Omega,\kappa}g^s)^{1/s}\in\mathcal{D}^p(u_t)$.
\end{lemma}

\begin{proof}
Let $u\in N^{1,p}_0(\Omega)$ and $g\in \mathcal{D}^p(u)$.
Since
${\M}_{\Omega,\kappa}g$ depends only on the values of $g$ in $\Omega$, 
by the gluing lemma \cite[Lemma~2.19]{BjornBjorn2011} we may assume that $g=0$ in $X\setminus \Omega$.
 
Take a Whitney cover ${\mathcal W}_c(\Omega)=\{ B_i:i\in \N\}$ with $c=t/18$, set $L=6$, 
and write $u_t(x)$ for $x\in X$ as
\[
u_t(x)=u(x)+\sum_{i=1}^{\infty}\varphi_i(x)\bigl(u_{B(x_i,3r_i)}-u(x)\bigr).
\]
By the properties of the Lipschitz functions $\varphi_i$, see above, and by the product rule, 
see for instance~\cite[Lemma~2.15]{BjornBjorn2011},
%or \cite[Proposition~6.3.28]{HeinonenKoskelaShanmugalingamTyson2015}, 
we have that
\[
\biggl(\frac{K}{r_i}|u-u_{B(x_i,3r_i)}|+g\biggr)\ch{B(x_i,6r_i)}
\]
belongs to $\mathcal{D}^p(\varphi_i(u_{B(x_i,3r_i)}-u))$. This implies that
\[
g+\sum_{i=1}^{\infty}\biggl(\frac{K}{r_i}|u-u_{B(x_i,3r_i)}|+g\biggr)\ch{B(x_i,6r_i)}\in \mathcal{D}^p(u_t).
\]

To estimate this $p$-weak upper gradient of $u_t$ in terms of ${\M}_{\Omega,\kappa}g^s$, we fix a Lebesgue point $x\in B(x_i,6r_i)$ of function $u$, 
notice that $B(x_i,3r_i)\sub B(x,9r_i)$, and write
\begin{equation}\label{Est}
|u(x)-u_{B(x_i,3r_i)}|\le |u(x)-u_{B(x,9r_i)}|+|u_{B(x,9r_i)}-u_{B(x_i,3r_i)}|.
\end{equation}
For the second term on the right-hand side of \eqref{Est} we have
\begin{align*}
|u_{B(x,9r_i)}-u_{B(x_i,3r_i)}|&\le\vint_{B(x_i,3r_i)}|u-u_{B(x,9r_i)}|\,d\mu %\\&
\le C\vint_{B(x,9r_i)}|u-u_{B(x,9r_i)}|\,d\mu\\
&\le C r_i\biggl(\vint_{B(x,9\lambda r_i)} g(y)^s \,d\mu(y)\biggr)^{\frac{1}{s}}\le Cr_i ({\M}_{\Omega,\kappa}g^s(x))^{1/s}.
\end{align*}
Here we use the fact that $g$ is also an $s$-weak upper gradient of $u$, and
the last estimate follows from the definition of maximal function ${\M}_{\Omega,\kappa}g^s$
and the inequality
\[
9\lambda r_i\le 9\lambda\biggl(\frac{18}{t}-6\biggr)^{-1}d(x,\Omega^c) < 9\lambda\biggl(\frac{9\lambda+6\kappa}{\kappa}-6\biggr)^{-1}d(x,\Omega^c) \le \kappa d(x,\Omega^c).
\]

For the first term on the right-hand side of \eqref{Est}, by a standard telescoping argument we have for
the Lebesgue point $x\in B(x_i,6r_i)$ of $u$ that
\begin{align*}
|u(x)-u_{B(x,9r_i)}|&\le C\sum_{j=0}^{\infty}\vint_{B(x,3^{2-j}r_i)}|u-u_{B(x,3^{2-j}r_i)}|\,d\mu\\ 
&\le C\sum_{j=0}^{\infty}3^{2-j}r_i\biggl(\vint_{B(x, \lambda 3^{2-j} r_i)} g(y)^s \,d\mu(y)\biggr)^{\frac{1}{s}}
\le Cr_i ({\M}_{\Omega,\kappa}g^s(x))^{1/s}.
\end{align*}
The last inequality is true for the same reason as above since $B(x,3^{2-j}\lambda r_i)\subset B(x,9\lambda r_i)$
for every $j=0,1,\dots$.

Hence, we obtain for every Lebesgue point $x\in B(x_i,6r_i)$ of $u$ that
\[
|u(x)-u_{B(x_i,3r_i)}|\le Cr_i ({\M}_{\Omega,\kappa}g^s(x))^{1/s}.
\]
This, together with the fact that $g=0$ in $X\setminus \Omega$, gives
\[
g(x)+\sum_{i=1}^{\infty}\biggl(\frac{K}{r_i}|u(x)-u_{B(x_i,3r_i)}|+g\biggr)\ch{B(x_i,6r_i)}(x)\le C({\M}_{\Omega,\kappa}g^s(x))^{1/s},
\]
for almost every $x\in X$. 
This, in turn, implies that $C({\M}_{\Omega,\kappa}g^s)^{1/s}$ is a $p$-weak upper gradient of $u_t$; see e.g.~\cite[Corollary 1.44]{BjornBjorn2011}. Notice that here the constant $C$ does not depend on $u$.
Finally, since ${\M}_{\Omega,\kappa}g^s\le {\M}g^s$
and $g\in \mathcal{D}^p(u)$, where $1\le s<p$,
the maximal function theorem (with exponent $p/s>1$) implies that
$({\M}_{\Omega,\kappa}g^s)^{1/s}\in L^p(X;d\mu)$.
Hence we conclude that $C({\M}_{\Omega,\kappa}g^s)^{1/s}\in\mathcal{D}^p(u_t)$,
as desired.
\end{proof}

\section{Quasiadditivity of capacity and Hardy--Sobolev inequalities}\label{sect:qadd}

This section contains our main results, which relate
weighted Hardy--Sobolev inequalities to the quasiadditivity
of the weighted capacities. 
The general approach in the proofs is similar
to those in~\cite{DydaVahakangas2015,HurriSyrjanenVahakangas2015,KLVbook,LuiroVahakangas2016},
but the present context requires different tools on the level of details. 
We first show that the validity of the 
$(q,p,\beta)$-Hardy--Sobolev inequality~\eqref{eq:(p,q)-weightedhardy} in $\Omega$,
for $1<p\le q<\infty$, implies the quasiadditivity property for 
the weighted relative capacity
with respect to  Whitney balls. 

\begin{thm}\label{t.necessary_qadd_weighted}
Let $1<p\le q<\infty$ and let $\Omega\subsetneq X$ be an open set. Assume that $(q,p,\beta)$-Hardy--Sobolev inequality \eqref{eq:(p,q)-weightedhardy} holds in $\Omega$ with a constant $C_1$, and let 
$\mathcal{W}_c(\Omega)=\{B_i:i\in\N\}$ be a Whitney cover of $\Omega$
for some $0<c<1/3$. Then there exists a constant $C$ such that
\[
\sum_{i\in \N} \capp{p,\beta}(E\cap B_i,\Omega)^{\frac{q}{p}} \leq  C \capp{p,\beta}(E,\Omega)^{\frac{q}{p}},
\]
for every set $E\Subset\Omega$.
\end{thm}

\begin{proof}
Let $E\Subset\Omega$.
Let $u\in N^{1,p}_c(\Omega)$ be such that $u(x)\ge 1$, for every $x\in E$, and let $g\in\mathcal{D}^p(u)$.

Fix $L>1$
satisfying $c<(3L)^{-1}$ and recall that $B_i^*=LB_i$, $i\in\N$.
For every $i\in\N$ we choose a Lipschitz function $\varphi_i$ satisfying the
following properties: $\varphi_i(x)=1$ for every $x\in B_i$, $\varphi_i$ is $K/r_i$-Lipschitz, for some $K\geq 1$, 
and $0\le \varphi_i(x)\le \ch{B_i^*}(x)$, for every $x\in X$. 
For instance, we can take
\[
\varphi_i(x)=\max\left\{0,1-\frac{d(x,B_i)}{d(B_i,X\setminus B_i^*)}\right\}, \qquad x\in X,
\]
in which case we can choose $K=(L-1)^{-1}$. 
Then, for every $i\in\N$, $u_i=u\varphi_i\in N^{1,p}_c(\Omega)$,
and the product rule %, see for instance, 
(see~\cite[Lemma~2.15]{BjornBjorn2011})
implies that the function
\[
g_{i}(x)=(g(x)+Kr_i^{-1}|u(x)|) \ch{B_i^*}(x),\quad x\in X,
\]
belongs to $\mathcal{D}^p(u_i)$.
Hence $u_i$ is a capacity test function
for $E\cap B_i$ and
\[
\capp{p,\beta}(E\cap B_i,\Omega)\le \int_\Omega g_{i}(x)^p\,d(x,\Omega^c)^{\beta} \,d\mu(x).
\]
Since $r_i^{-1}\le C d(x,\Omega^c)^{-1}$, for every $x\in B_i^*$,
\begin{align*}
&\biggl(\int_\Omega g_{i}(x)^p\,d(x,\Omega^c)^{\beta}\,d\mu(x)\biggr)^\frac{q}{p}\\
&\qquad\le C \biggl(\int_{B_i^*}g(x)^p\,d(x,\Omega^c)^{\beta}\,d\mu(x)\biggr)^\frac{q}{p}
+C\biggl(\int_{B_i^*}|u(x)|^p d(x,\Omega^c)^{-p+\beta}\,d\mu(x)\biggr)^\frac{q}{p}.
\end{align*}
Since $q\geq p$, we can write
\begin{align*}
&\sum_{i=1}^{\infty} \capp{p,\beta}(E\cap B_i,\Omega)^{\frac{q}{p}}
\le \sum_{i=1}^{\infty} \biggl(\int_\Omega g_{i}(x)^p\,d(x,\Omega^c)^{\beta}\,d\mu(x)\biggr)^\frac{q}{p}\\ 
&\qquad \le C\sum_{i=1}^{\infty}  \biggl(\int_{B_i^*}g(x)^p\,d(x,\Omega^c)^{\beta}\,d\mu(x)\biggr)^\frac{q}{p}
+C\sum_{i=1}^{\infty} \biggl(\int_{B_i^*}|u(x)|^p d(x,\Omega^c)^{-p+\beta}\,d\mu(x)\biggr)^\frac{q}{p} \\
&\qquad \le   C\Biggl(\sum_{i=1}^{\infty}\int_{B_i^*}g(x)^p\,d(x,\Omega^c)^{\beta}\,d\mu(x)\Biggr)^\frac{q}{p}
+C\sum_{i=1}^{\infty}\mu(B_i^*)^{\frac{q}{p}\frac{q-p}{q}}\int_{B_i^*}\frac{|u(x)|^q}{d(x,\Omega^c)^{\frac{q}{p}(p-\beta)}}\,d\mu(x).
\end{align*}
For the first term on the right-hand side, due to the finite overlap of the balls $B_i^*$, we get
\[
\Biggl(\sum_{i=1}^{\infty}\int_{B_i^*}g(x)^p\,d(x,\Omega^c)^{\beta}\,d\mu(x)\Biggr)^\frac{q}{p}\le
C\biggl(\int_{\Omega}g(x)^p\,d(x,\Omega^c)^{\beta}\,d\mu(x)\biggr)^\frac{q}{p}.
\]
To estimate the second term, we notice from Lemma~\ref{whitney} that $B_i^*\subset B(x,d(x,\Omega^c))$, for every $x\in B_i^*$. This inclusion together with the finite overlap of the balls $B_i^*$ and
inequality~\eqref{eq:(p,q)-weightedhardy} gives
\begin{align*}
&\sum_{i=1}^{\infty}\mu(B_i^*)^{\frac{q-p}{p}}\int_{B_i^*}\frac{|u(x)|^q}{d(x,\Omega^c)^{\frac{q}{p}(p-\beta)}}\,d\mu(x)\\ 
&\qquad\le C\sum_{i=1}^{\infty}\int_{B_i^*} \frac{|u(x)|^{q}}{d(x,\Omega^c)^{\frac{q}{p}(p-\beta)}}\mu(B(x,d(x,\Omega^c)))^{\frac{q-p}{p}}\,d\mu(x)\\
&\qquad \le  C\int_{\Omega}\frac{|u(x)|^q}{d(x,\Omega^c)^{\frac{q}{p}(p-\beta)}}\mu(B(x,d(x,\Omega^c)))^{\frac{q-p}{p}}\,d\mu(x)\\
&\qquad\le C\biggl(\int_{\Omega}g(x)^p\,d(x,\Omega^c)^{\beta}\,d\mu(x)\biggr)^\frac{q}{p}.
\end{align*}
After collecting the estimates above, and taking the infimum first over all $g\in\mathcal{D}^p(u)$
and then over all functions $u$ as above, the claim follows.
\end{proof}

For $q\le \frac{Qp}{Q-p}$ and sufficiently small parameters $c$, there is also a 
partial converse of Theorem~\ref{t.necessary_qadd_weighted}. In addition
to the quasiadditivity property in part~\ref{it.weightedhardy_c_compact},
the next result contains the \emph{weak quasiadditivity} property
in part~\ref{it.weightedhardy_c_cubes}. 

\begin{thm}\label{t.weightedhardy_c}
Let $1<p\le q\le \frac{Qp}{Q-p}$ and 
assume that $X$ supports a $(1,s)$-Poincar\'{e} inequality for some $1<s<p$ with a dilation constant $\lambda$.
Let $\Omega\subsetneq X$ be an open set and let $\mathcal{W}_c(\Omega)=\{B_i:i\in\N\}$ be a
Whitney cover of $\Omega$ with $0<c<\frac{1}{45\lambda+8}$.
Then the following conditions are equivalent:
\begin{enumerate}[label=\textup{(\roman*)}] %,leftmargin=18pt

\item\label{it.weightedhardy_c_HS} The $(q,p,\beta)$-Hardy--Sobolev inequality \eqref{eq:(p,q)-weightedhardy} holds in $\Omega$.

\item\label{it.weightedhardy_c_compact} There exist constants $C_1$ and $C_2$ such that 
\begin{equation*}
\sum_{i=1}^\infty \capp{p,\beta}(E\cap  B_i,\Omega)^{\frac q p}\le C_1 \capp{p,\beta}(E,\Omega)^{\frac q p},
\end{equation*}
for every set $E\Subset\Omega$,
and the capacity lower bound \eqref{eq.weightedcap_lower_w} holds with the constant $C_2$
for every ball $B_i\in \mathcal{W}_c(\Omega)$.

\item\label{it.weightedhardy_c_cubes} There exist constants $C_1$ and $C_2$ such that 
\begin{equation*}
\sum_{i\in I} \capp{p,\beta}( B_{i},\Omega)^{\frac q p}\le
C_1 \capp{p,\beta}\Biggl(\bigcup_{i\in I} B_{i},\Omega\Biggr)^{\frac q p},
\end{equation*}
whenever $I\subset\N$ is a finite set,
and the capacity lower bound \eqref{eq.weightedcap_lower_w} holds with the constant $C_2$
for every ball $B_i\in \mathcal{W}_c(\Omega)$.
\end{enumerate}
\end{thm}

\begin{proof}
The implication from~\ref{it.weightedhardy_c_HS} to~\ref{it.weightedhardy_c_compact} follows from 
Lemma~\ref{eq.LowerBallweightedCapacity} and
Theorem~\ref{t.necessary_qadd_weighted}.
The implication from~\ref{it.weightedhardy_c_compact} to~\ref{it.weightedhardy_c_cubes}
follows
by considering $E=\bigcup_{i\in I} B_{i}$.

Assume now that condition~\ref{it.weightedhardy_c_cubes} holds with a constant $C_1$.
To show that~\ref{it.weightedhardy_c_HS} holds,
by Theorem~\ref{t.mazya_char_weighted_qp} it suffices
to prove that there exists a constant $C$ such that
\begin{equation}\label{e.maz'ya_app}
\int_E \frac{\mu(B(x,d(x,\Omega^c)))^\frac{q-p}{p}}{d(x,\Omega^c)^{\frac{q}{p}(p-\beta)}}\,d\mu(x) 
\le C\capp{p,\beta}(E,\Omega)^{\frac q p},
\end{equation}
for every set $E\Subset\Omega$.
To this end, fix $E\Subset\Omega$,  
let $u\in N^{1,p}_c(\Omega)$ be such that $u(x)\ge 1$ for every $x\in E$, and let
$g\in\mathcal{D}^p(u)$. 
By considering $\lvert u\rvert$ instead of $u$, we may assume that $u\ge 0$ in $X$;
observe that $\lvert u\rvert\in N^{1,p}_c(\Omega)$ and 
$g\in\mathcal{D}^p(\lvert u\rvert)$.

Partition $\mathcal{W}_c(\Omega)=\{B_i : i\in\N\}$ into two subfamilies:
$\mathcal{W}^1 = \{B_i\in\mathcal{W}_c(\Omega) : u_{B_i}  < \tfrac12\}$
and $\mathcal{W}^2 = \mathcal{W}_c(\Omega)\setminus \mathcal{W}^1$.
Observe that $\mathcal{W}^2$ is necessarily finite since $u\in N^{1,p}_c(\Omega)$. 
The left-hand side of~\eqref{e.maz'ya_app} can be estimated from above by
\begin{equation}\label{e.w_split}
\Bigg(\sum_{B\in\mathcal{W}^1} + \sum_{B\in\mathcal{W}^2}\Bigg)
\int_{E\cap B}\frac{\mu(B(x,d(x,\Omega^c)))^\frac{q-p}{p}}{d(x,\Omega^c)^{\frac{q}{p}(p-\beta)}}\,d\mu(x).
\end{equation}
To estimate the first sum, we
observe that, for every $B_i\in\mathcal{W}^1$ and every $x\in E\cap B_i$,
\[
\tfrac 12 =  1-\tfrac 12 < u(x) - u_{B_i}  = \lvert u(x)- u_{B_i}\rvert.
\]
By the definition and the properties of Whitney balls, see Lemma~\ref{whitney}, if $B_i\in\mathcal{W}_c(\Omega)$ and $x\in B_i$, then
\[
B_i\subset B(x, d(x, \Omega^c))\quad\text{ and }\quad (1/c-1)r_i\le d(x, \Omega^c)\le (1/c+1)r_i.
\]
Consequently,
\begin{equation}\label{MeasuresRelations}
\mu(B_i)\le \mu(B(x,d(x,\Omega^c)))\quad \text{ and }\quad \mu(B(x,d(x,\Omega^c)))\le C \mu(B_i)
\end{equation}
for every $x\in B_i$.

Since $X$ supports a $(1,s)$-Poincar\'e inequality for $1<s<p$ with the dilation constant $\lambda$, then $X$ supports a $(1,p)$-Poincar\'e inequality with the same dilation constant $\lambda$, and hence the $(q,p)$-Poincar\'e inequality with the dilation constant $2\lambda$; see \cite[Theorem~4.21]{BjornBjorn2011}.
 
The observations above lead to the estimates
\begin{align*}
\sum_{B\in\mathcal{W}^1} \int_{E\cap B}  \frac{\mu(B(x,d(x,\Omega^c)))^{\frac{q-p}{p}}}{d(x,\Omega^c)^{\frac{q}{p}(p-\beta)}}\,d\mu (x)
&\le C \sum_{B_i\in\mathcal{W}^1} \frac{\big(\mu(B_i)r_i^\beta\big)^{\frac{q}{p}}}{r_i^{q}} \vint_{B_i} \lvert u(x)- u_{B_i}\rvert^q\,d\mu (x)\\
& \le C \sum_{B_i\in\mathcal{W}^1} \frac{\big(\mu(B_i)r_i^\beta\big)^{\frac{q}{p}}}{r_i^{q}} r_i^{q} \biggl(\vint_{2\lambda B_i} g(x)^p \,d\mu(x)\biggr)^{\frac q p}\\
& \le C  \Biggl(\sum_{B\in\mathcal{W}^1}\int_{2\lambda B} g(x)^p\, d(x, \Omega^c)^\beta\,d\mu(x)\biggr)^{\frac q p}\\
& \le C \biggl(\int_{\Omega} g(x)^p \,d(x, \Omega^c)^\beta\,d\mu(x)\biggr)^{\frac q p}.
\end{align*}
The last inequality follows from the finite overlap of dilated Whitney balls $2\lambda B$ which is guaranteed by the relation between $c$ and $\lambda$ in the formulation of the theorem.

To estimate the second sum in~\eqref{e.w_split},
let $B_i=B(x_i,r_i)\in\mathcal{W}^2$ and $x\in B_i$. Recall that $r_i=c d(x_i,\Omega^c)$, 
where $c<\tfrac{1}{45\lambda+8}<\tfrac{1}{20}$.
Choose a number $t$ such that 
$\frac{18c}{1-2c}<t<\frac{6}{15\lambda+2}$; then, in particular, $0<t<1$.
Let $\mathcal{W}_{t/18}(\Omega)=\{B(x'_j,r'_j) : j\in\N\}$ be a Whitney cover of $\Omega$, 
$\{\varphi_j\}$ be a partition of unity related to $\mathcal{W}_{t/18}(\Omega)$, 
and $u_t$ be the the discrete convolution of $u\ge 0 $ at scale $t$ as
in \eqref{DiscreteConvolution}. Then $x\in B(x'_j,r'_j)$ for some $j\in \N$ 
and the choice of $t$ guarantees that $r_i\le r'_j$, since by property (3) of Lemma~\ref{whitney} we have
\[
(1/c-1)r_i\le d(x,\Omega^c)\le (18/t+1)r'_j,
\]
which implies
\[
r_i\le\frac{18/t+1}{1/c-1}r'_j\le r'_j.
\]
Similarly, we obtain that $r_j'\le Cr_i$, for some $C>0$, since
\[
(18/t-1)r'_j\le d(x,\Omega^c)\le (1/c+1)r_i.
\] 
By the definition of the discrete convolution $u_t$, properties of functions $\varphi_j$, 
the inclusion 
$B(x_i,r_i)\subset B(x'_j,3r'_j)$ (from $r_i\le r'_j$), 
and the doubling condition for measure $\mu$, we obtain
\begin{equation}\label{e.m_est1}
\begin{split}
u_t(x) & \ge \varphi_j(x)\vint_{B(x'_j,3r'_j)} u(y)\,d\mu(y)
\ge C\vint_{B(x'_j,3r'_j)} u(y)\,d\mu(y) \\ 
& \ge C \vint_{B(x_i,r_i)} u(y)\,d\mu(y) \ge C u_{B_i} \ge \frac{C}{2},
\end{split}
\end{equation}
%where $C$ depends on $\kappa$, $\lambda$, $c$, and the doubling constant of the measure $\mu$.
where $C$ depends on $\lambda$, $c$, and the doubling constant of the measure $\mu$.

Since $u\in N^{1,p}_c(\Omega)$, we have that $u\in L^p(X;d\mu)$, and hence $u_t\in L^p(X;d\mu)$; 
see the comment before Lemma~\ref{MaxFunctIsUpperGradient}. Furthermore, by Lemma~\ref{MaxFunctIsUpperGradient} the set $\mathcal{D}^p(u_t)$ is nonempty, and
as the sum in \eqref{DiscreteConvolution} has only finitely many non-zero terms for $u\in N^{1,p}_c(\Omega)$, the support of $u_t$
is bounded and has a positive distance to $\Omega^c$. Thus we conclude that $u_t\in N^{1,p}_c(\Omega)$.

By~\eqref{e.m_est1} there exists $C_2>0$ such that $C_2u_t\ge 1$ in $\bigcup_{B\in\mathcal{W}^2} B$.
Hence, $C_2 u_t$ is a capacity test function for
$\capp{p,\beta}\bigl(\bigcup_{B\in\mathcal{W}^2} B,\Omega\bigr)$. Since $t<6/(15\lambda+2)$, we can choose
$0<\kappa < 1/5$ such that $t<6\kappa/(3\lambda+2\kappa)$, and, by 
Lemma~\ref{MaxFunctIsUpperGradient}, there is a constant $C_3$ such that 
$C_3({\M}_{\Omega,\kappa}g^s)^{1/s}\in \mathcal{D}^p(u_t)$.

Let $B_i=B(x_i,r_i)\in \mathcal{W}^2$.
By inequalities \eqref{MeasuresRelations} and \eqref{eq.weightedcap_lower_w}, it follows that
\begin{align*}
\int_{E\cap B_i} \frac{\mu(B(x,d(x,\Omega^c)))^{\frac{q-p}{p}}}{d(x,\Omega^c)^{\frac{q}{p}(p-\beta)}}\,d\mu(x)
& \le C\mu(B_i)^{\frac{q}{p}-1} \int_{E\cap B_i}d(x,\Omega^c)^{\frac{q}{p}(\beta-p)}\,d\mu(x)\\
& \le C\big(\mu(B_i)r_i^{\beta-p}\big)^{q/p}\le C\capp{p,\beta}(B_i,\Omega)^{q/p}.
\end{align*}
Since $\mathcal{W}^2$ is finite, we obtain
from the assumed condition~\ref{it.weightedhardy_c_cubes}
that
\begin{align*}
\sum_{B\in\mathcal{W}^2}
\int_{E\cap B} \frac{\mu(B(x,d(x,\Omega^c)))^\frac{q-p}{p}}{d(x,\Omega^c)^{\frac{q}{p}(p-\beta)}}\,d\mu(x)
&\le C \sum_{B\in\mathcal{W}^2} \capp{p,\beta}(B,\Omega)^{\frac q p}\\
&\le C \capp{p,\beta}\Biggl(\bigcup_{B\in\mathcal{W}^2} B,\Omega\Biggr)^{\frac q p}\\
&\le C \biggl(\int_{\Omega} ({\M}_{\Omega,\kappa} g^s(x))^{p/s}\,d(x, \Omega^c)^\beta\,d\mu(x)\biggr)^{\frac q p}\\
&\le C \biggl(\int_{\Omega} g(x)^p\,d(x, \Omega^c)^\beta\,d\mu\biggr)^{\frac q p},
\end{align*}
where the last inequality is valid since $s<p$, and thus the maximal operator ${\M}_{\Omega,\kappa}$ is bounded on $L^{p/s}(\Omega,w_{\beta}\,d\mu)$ by Lemma~\ref{l.maximal_bounded}.

By combining the estimates for $\mathcal{W}^1$ and $\mathcal{W}^2$
we obtain
\[\begin{split}
\int_E \frac{\mu(B(x,d(x,\Omega^c)))^\frac{q-p}{p}}{d(x,\Omega^c)^{\frac{q}{p}(p-\beta)}}\,d\mu(x)
& \le \sum_{B\in\mathcal{W}} \int_{E\cap B}\frac{\mu(B(x,d(x,\Omega^c)))^\frac{q-p}{p}}{d(x,\Omega^c)^{\frac{q}{p}(p-\beta)}}\,d\mu(x)\\
&\le C\biggl(\int_{\Omega}g(x)^p\,d(x, \Omega^c)^\beta\,d\mu\biggr)^{\frac q p}.
\end{split}\]
The desired estimate~\eqref{e.maz'ya_app} follows
by taking infimum over all $g\in\mathcal{D}^p(u)$ and then over all functions
$u$ as above.
\end{proof}

Theorem~\ref{t.weightedhardy_c} gives the 
following implication for the validity
of $(q,p,\beta)$-Hardy--Sobolev inequalities for different values of
the parameter $q$. Observe that this is not completely obvious from the
statement of the Hardy--Sobolev inequality.

\begin{remark}\label{r.qadd_interpolate}
Let $1<p\le q<p^*=\frac{Qp}{Q-p}$ and $\beta\in\R$, and
assume that $X$ supports a $(1,s)$-Poincar\'{e} inequality for some $1<s<p$.
If a $(q,p,\beta)$-Hardy--Sobolev inequality holds in an open set $\Omega\subsetneq X$,
then also $(q',p,\beta)$-Hardy--Sobolev inequalities hold in $\Omega$ for every $q\le q' \le p^*$.
This follows from Theorem~\ref{t.weightedhardy_c} and the fact that if $q\le q'$, then
\[
\Biggl(\sum_{i=1}^\infty \capp{p,\beta}(E\cap B_i,\Omega)^{\frac {q'} p}\Biggr)^{\frac p {q'}}
\le \Biggl(\sum_{i=1}^\infty \capp{p,\beta}(E\cap B_i,\Omega)^{\frac q p}\Biggr)^{\frac p q}
\] 
whenever $E\subset \Omega$ and 
$\mathcal{W}_c(\Omega)=\{B_i:i\in\N\}$ is a
Whitney cover of $\Omega$.
%, satisfying the assumptions in Theorem~\ref{t.weightedhardy_c}.
\end{remark}

\section{Special cases and examples}\label{sect:qreg}

The main results of the paper are obtained under the assumption that the measure $\mu$ is doubling. In particular, 
$\mu$ satisfies for some $Q>0$ the inequality
\[
\frac{\mu(B(y,r))}{\mu(B(x,R))}\ge C\Bigl(\frac rR\Bigr)^Q,
\]
whenever $0<r\le R<\diam X$ and $y\in B(x,R)$. 

If $X$ is connected
(this is guaranteed in our setting by the Poincar\'e inequalities), then there exists
also a constant $Q_u$ satisfying $0<Q_u\le Q$ and such that
\begin{equation}\label{eq:QuReverseDoubling}
   \frac{\mu(B(y,r))}{\mu(B(x,R))}\le C\Bigl(\frac rR\Bigr)^{Q_u},
\end{equation}
whenever $0<r\le R<\diam X$ and $y\in B(x,R)$; see e.g.~\cite[Corollary 3.8]{BjornBjorn2011}.

If there are uniform upper and lower bounds for the measures of
the balls, that is,
\[
C^{-1} r^Q\leq \mu(B(x,r))\leq C r^Q
\]
for every $x\in X$ and all $0<r<\diam(X)$, the measure $\mu$ is said
to be \emph{Ahlfors $Q$-regular}.
The above exponent $Q>0$ plays the same role as $n$ does in $\mathbb{R}^n$. 
In the $Q$-regular case
the $(q,p,\beta)$-Hardy--Sobolev inequality \eqref{eq:(p,q)-weightedhardy} can be written in a simpler form
\begin{equation}\label{eq:(p,q)-weightedhardyQregularSpace}
\biggl(\int_{\Omega} |u(x)|^q d(x,\Omega^c)^{\frac{q}{p}(Q-p+\beta)-Q} \,d\mu(x)\biggr)^{\frac{1}{q}}
   \leq C\biggl(\int_{\Omega} g(x)^p \,d(x,\Omega^c)^{\beta}\,d\mu\biggr)^{\frac{1}{p}},
\end{equation}
and Theorem~\ref{t.weightedhardy_c} implies the following characterization.
In $X=\R^n$ (equipped with the usual Euclidean distance and the Lebesgue measure), 
the corresponding result in terms of Whitney cubes has been considered in~\cite[Theorem~10.52]{KLVbook}
in the unweighted case $\beta=0$,
but the weighted result is new  even in $\R^n$.
Since $N^{1,p}(\R^n)=W^{1,p}(\R^n)$ and $\R^n$ supports $(1,s)$-Poincar\'e inequalities for every $s\ge 1$ with $\lambda=1$,
Theorem~\ref{t.qadd_intro} follows from Corollary~\ref{c.weightedhardy_c_Qreg}.

\begin{coro}\label{c.weightedhardy_c_Qreg}
Let $X$ be a metric space equipped with an Ahlfors $Q$-regular measure $\mu$ and let $1<p\le q\le \frac{Qp}{Q-p}$.
Assume that $X$ supports a $(1,s)$-Poincar\'{e} inequality for some $1<s<p$ with a dilation constant $\lambda$. Let $\mathcal{W}_c(\Omega)=\{B_i:i\in\N\}$, $B_i=(x_i,r_i)$, be a
Whitney cover of an open set $\Omega\subsetneq X$,
with $0<c<\frac{1}{45\lambda+8}$.
Then the following conditions are equivalent:
\begin{enumerate}[label=\textup{(\roman*)}] %,leftmargin=18pt

\item\label{it.weightedhardy_c_HS QR} Hardy--Sobolev inequality \eqref{eq:(p,q)-weightedhardyQregularSpace} holds in $\Omega$.

\item\label{it.weightedhardy_c_compact QR} There exist constants $C_1$ and $C_2$ such that
\begin{equation*}
\sum_{i=1}^\infty \capp{p,\beta}(E\cap  B_i,\Omega)^{\frac q p}\le C_1 \capp{p,\beta}(E,\Omega)^{\frac q p},
\end{equation*}
for every set $E\Subset\Omega$, and $\capp{p,\beta}(B_i,\Omega) \ge C_2 r_i^{Q+\beta-p}$ for every $i\in \N$. 

\item\label{it.weightedhardy_c_cubes QR} There exist constants $C_1$ and $C_2$ such that
\begin{equation*}
\sum_{i\in I} \capp{p,\beta}( B_{i},\Omega)^{\frac q p}\le
C_1 \capp{p,\beta}\Biggl(\bigcup_{i\in I} B_{i},\Omega\Biggr)^{\frac q p},
\end{equation*}
whenever $I\subset\N$ is a finite set, and $\capp{p,\beta}(B_i,\Omega) \ge C_2 r_i^{Q+\beta-p}$ for every $i\in \N$. 
\end{enumerate}
\end{coro}

When $1<p<Q$ and $\beta=0$, 
the capacity lower bound
$\capp{p}(B,\Omega)\geq C r^{Q-p}$
holds for all Whitney balls
$B=B(x,r)\in\W_c(\Omega)$, at least if
$X$ is unbounded or $\diam(\Omega) < \frac 1 4\diam(X)$;
see~\cite[Lemma~2]{LehrbackShanmugalingam2014}
and~\cite[Proposition~6.1]{BjornBjornLehrback2017} for details. 
Here $\capp{p}(B,\Omega)= \capp{p,0}(B,\Omega)$.
(More generally, the capacity lower bound~\eqref{eq.weightedcap_lower_w}
holds in the unweighted case $\beta=0$ if $1<p<Q_u$, by the same references.)

Consequently, in certain cases, 
for instance if $X$ is an unbounded $Q$-regular space
and $1<p<Q$, $\beta=0$, %in the unweighted case with $1<p<Q$, 
there is no need to assume %inequality \eqref{eq:low est for cap} 
the capacity lower bound in conditions \ref{it.weightedhardy_c_compact QR} and \ref{it.weightedhardy_c_cubes QR} in Corollary~\ref{c.weightedhardy_c_Qreg}. This also shows that
the weighted capacity lower bound~\eqref{eq.weightedcap_lower_w}, for every Whitney ball, is not alone sufficient for the 
$(q,p,\beta)$-Hardy--Sobolev inequality in Theorem~\ref{t.weightedhardy_c}. 
For instance, consider $X=\R^n$ equipped with the usual Euclidean distance and the Lebesgue measure.
Then $Q=n$, and
for every $1<p<n$ and $p\le q < \frac{np}{n-p}$ 
there are open sets where the $(q,p,0)$-Hardy--Sobolev inequality does not hold.
When $1<p=q<n$, one such an example is given by the complement of any (non-empty) Ahlfors $(n-p)$-regular set,
by~\cite[Theorem 1.1]{KoskelaZhong2003}; 
see also~\cite[Example 4.7]{LehrbackVahakangas2016} for counterexamples in more general cases. 

On the other hand, even if the capacity lower bound~\eqref{eq.weightedcap_lower_w}
is not always needed,
it is not possible to remove this part
from conditions \ref{it.weightedhardy_c_compact QR} and \ref{it.weightedhardy_c_cubes QR} in Theorem~\ref{t.weightedhardy_c}.
This is illustrated by the next example.

\begin{example}
Consider $X=\R^n$, equipped with the usual Euclidean distance and the Lebesgue measure. 
Then $N^{1,p}(\R^n)=W^{1,p}(\R^n)$, and $\lvert \nabla u\rvert\in \mathcal{D}_p(u)$ if
$u\in W^{1,p}(\R^n)$;  see e.g.~\cite[A.1 and Proposition A.13]{BjornBjorn2011}.

Let $1<p<\infty$ and $\beta\ge 0$,
and choose $\Omega=B(0,1)\subset \R^n$. For every $j\in\N$, let
$u_j$ be a Lipschitz continuous function in $\R^n$,
satisfying the following properties:
$u_j=1$ in $B(0,1-2^{-j})$, 
$u_j=0$ in $\R^n\setminus B(0,1-2^{-j-1})$, 
and $\lvert \nabla u\rvert\le 2^{j+1}$ almost everywhere in
$A_j=B(0,1-2^{-j-1})\setminus B(0,1-2^{-j})$.
Then we have for every $j\in\N$ that
$u_j\in W_0^{1,p}(\Omega)$ 
has a compact support in $\Omega$ and
\begin{equation*}
\int_{\Omega} \lvert \nabla u_j(x)\rvert^p \,d(x,\Omega^c)^{\beta}\,d\mu
\le \mu(A_j)2^{(j+1)p}2^{-j\beta} \le C 2^{-j(1-p+\beta)}.
\end{equation*}
If $\beta>p-1$, then $2^{-j(1-p+\beta)}\to 0$ as $j\to\infty$.

Let $\mathcal{W}_c(\Omega)=\{B_i:i\in\N\}$ be a
Whitney cover of $\Omega$.
When $i\in\N$ is fixed,
the functions $u_j$ are test functions
for the capacity $\capp{p,\beta}(B_i,\Omega)$ for all sufficiently large $j$.
Thus the above computation shows that if $\beta > p-1$, 
then $\capp{p,\beta}(B_i,\Omega)=0$ for all Whitney balls $B_i\in\mathcal{W}_c(\Omega)$,
and so the capacity lower bound~\eqref{eq.weightedcap_lower_w} does not hold in this case.
Lemma~\ref{eq.LowerBallweightedCapacity} (or a direct computation for the functions $u_j$) 
implies that neither does the $(q,p,\beta)$-Hardy--Sobolev inequality hold when $\beta>p-1$.
Nevertheless, the quasiadditivity property
\[
\sum_{i\in \N} \capp{p,\beta}(E\cap B_i,\Omega)^{\frac{q}{p}} \leq  C \capp{p,\beta}(E,\Omega)^{\frac{q}{p}}
\]
holds trivially
for every set $E\Subset\Omega$, since all the terms on the left-hand side are zero.
This shows that 
in addition to the quasiadditivity property, also
the capacity lower bound is really needed in assertions
\ref{it.weightedhardy_c_compact QR} and~\ref{it.weightedhardy_c_cubes QR}
of Theorem~\ref{t.weightedhardy_c}
and in the corresponding assertions of Corollary~\ref{c.weightedhardy_c_Qreg}.
\end{example}

%\bibliographystyle{abbrv}
%\bibliography{qadd}

\end{document}